\documentclass[a4paper]{amsart}
\usepackage[pdftex]{graphicx}
\usepackage{amsmath,amssymb,amsthm}
\usepackage{amsfonts}
\usepackage{mathrsfs}
\usepackage{mathtools}
\usepackage{hyperref}
\usepackage{xcolor}
\usepackage[T1]{fontenc}
\usepackage{tikz}
\usetikzlibrary{arrows.meta}

\theoremstyle{definition}
\newtheorem{thm}{Theorem}[section]
\newtheorem{prop}[thm]{Proposition}
\newtheorem{lem}[thm]{Lemma}
\newtheorem{cor}[thm]{Corollary}
\newtheorem{dfn}[thm]{Definition}
\newtheorem{rem}[thm]{Remark}
\newtheorem{prob}[thm]{Problem}
\newtheorem{exam}[thm]{Example}

\newcommand{\C}{\mathbb{C}}
\newcommand{\R}{\mathbb{R}}
\newcommand{\Z}{\mathbb{Z}}
\newcommand{\Q}{\mathbb{Q}}
\newcommand{\A}{\mathbb{A}}
\newcommand{\vect}{\mathbf{v}}
\newcommand{\head}{\mathbf{h}}
\newcommand{\tail}{\mathbf{t}}
\newcommand{\Spec}{\operatorname{Spec}}
\newcommand{\FB}{\operatorname{FB}}
\newcommand{\Hilb}{\operatorname{Hilb}}

\newcommand{\Crit}{\operatorname{Crit}}

\newcommand{\Vol}{\operatorname{Vol}}
\newcommand{\Star}{\operatorname{Star}}
\newcommand{\GL}{\operatorname{GL}}
\newcommand{\SL}{\operatorname{SL}}


\begin{document}
\title{Exceptional loci of F-blowups and $G$-Hilbert schemes}
\author{Enrique Chávez-Martínez, Yutaro Kaijima and Takehiko Yasuda}

\address{Instituto de Ingeniería y Tecnología, UACJ, Ciudad Juárez, Chihuahua,
México}
\email{enrique.chavez@uacj.mx}

\address{Department of Mathematics, Graduate School of Science, The University of Osaka Toyonaka, Osaka 560-0043, JAPAN}
\email{u797381f@ecs.osaka-u.ac.jp}

\address{Department of Mathematics, Graduate School of Science, the University
of Osaka, Toyonaka, Osaka 560-0043, Japan}
\address{Kavli Institute for the Physics and Mathematics of the Universe, The University of Tokyo, 5-1-5 Kashiwanoha, Kashiwa, Chiba, 277-8583, Japan}
\email{yasuda.takehiko.sci@osaka-u.ac.jp}
\maketitle

\begin{abstract}
We study exceptional loci of F-blowups of normal toric varieties. In the $\Q$-factorial case, this study amounts to studying the exceptional loci of $G$-Hilbert schemes. We give a formula for the dimension of the center of a prime divisor on the F-blowup in terms of combinatorial data, together with an algorithm for computing it. Moreover, we study the relation between F-blowups and essential divisors for three-dimensional terminal singularities and canonical singularities. Finally, we give a simple condition ensuring that a prime divisor over the given toric variety has a positive-dimensional center on the F-blowup.
\end{abstract}

\section{Introduction}


%
%

The primary objective of this paper is to understand the exceptional set of the F-blowup, associated to a $\Q$-factorial affine toric variety. To a toric variety $X$, we can associate another (possibly non-normal) toric variety $\FB_{(\infty)}(X)$ in a canonical way, called the limit F-blowup \cite{Yas12}. It is defined as the universal birational flattening of a Frobenius-like endomorphism of $X$. If $X$ is $\Q$-factorial and affine, equivalently, if $X$ is associated to  a simplicial cone, then $X$ is a quotient variety by a finite abelian group and $\FB_{(\infty)}(X)$ is isomorphic to the $G$-Hilbert scheme.  In some special cases, this birational transform provides a nice resolution of singularities of the given toric variety such as minimal resolutions and crepant resolutions. However, it is not smooth or even normal, in general.  


Let us denote the normalization of $\widetilde{\FB}_{(\infty)}(X)$ by $\widetilde{X}$. 
Besides the problem of smoothness, it is also interesting to study the geometry of the exceptional locus in $\widetilde{X}$.
 Recently, the first and third named authors together with Duarte proposed the following problem in this direction: 


\begin{prob}[{\cite[Problem 1.1]{CMDY25}}]\label{FB-ess}
Do all essential divisors over $X$ appear on $\widetilde{X}$ as prime divisors ?
\end{prob}

An essential divisor means a divisor over $X$ whose center on every desingularization of $X$ is an irreducible component of the exceptional locus. This notion appears in the famous Nash problem on arc families. We refer the reader to \cite{PS15} for more details on the Nash problem. 
Problem \ref{FB-ess} itself is also regarded as an analog of the Nash problem, since the F-blowup can be interpreted as a parameter space of some sort of jets.


\begin{rem}
A similar problem for higher Nash blowups was addressed in \cite{CM23}. There are also studies on toric resolutions that uses only essential divisors  \cite{BGS95,Dai02,SS23,Sat25}.
\end{rem}

In \cite{CMDY25}, it was proved that if $X$ admits a special type of resolution called a moderate toric resolution, then Problem \ref{FB-ess} has a positive answer.\footnote{Precisely speaking, the mentioned result in  \cite{CMDY25} uses the notion of BGS essential divisors, a variant of essential divisors. For $\Q$-factorial toric varieties,  BGS essential divisor is the same as the usual essential divisor.} However, as shown by the second named author and Yamamoto \cite{KY26},  a toric variety does not generally admit a moderate toric resolution. For example, a three-dimensional $\Q$-factorial toric variety with only terminal singularities has no moderate toric resolution. Thus, it is desirable to have a description of the exceptional locus that holds true even when a moderate toric resolution does not exist.

%

In this paper, we give several results on the geometry of exceptional loci in normalized limit F-blowups $\widetilde{X}$ with particular emphasis on the case of  $\Q$-factorial toric varieties. Note that our results on this case are also regarded as results on the $G$-Hilbert scheme,  since $\widetilde{X}$ is isomorphic to the normalization of the $G$-Hilbert scheme as mentioned above. 
To state our main results, we now set up some notation.
 Let $M$ and $N$ be free abelian groups of rank $d$ which are dual to each other. Let $\sigma\subset N_\R=N\otimes \R$ be a strongly convex, nondegenerate, rational polyhedral cone and let $X$ be the associated affine toric variety. 
The normalized limit F-blowup $\widetilde{X}$ corresponds to a fan $\Delta$ that is a subdivision of $\sigma$. Essential divisors over $X$ corresponds to the elements of the Hilbert basis of the monoid $\sigma\cap N$ other than the primitive generators of the one-dimensional faces of $\sigma$.  
With this notation, 
Problem \ref{FB-ess} is rephrased  as follows: 
Does the ray $\R_{\geq0}w$ for every element $w$ of the Hilbert basis  belong to the fan $\Delta$?

The key notion for our results is the one of $w$-critical arrows introduced in \cite{CMDY25}. 
For $0\neq w\in\sigma\cap N$, a $w$-critical arrow is an ordered pair $\alpha=(\head_\alpha,\tail_\alpha)$ of two points $\head_\alpha,\tail_\alpha$ in the dual cone $\sigma^\vee\subset M_\R$ such that the associated vector $\vect_\alpha=\head_\alpha-\tail_\alpha$ is an element of $M\cap w^\perp$ and some extra conditions are satisfied. 
Let  $V_w \subset M_\R$  be the subspace spanned by the vectors associated to  $w$-critical arrows. The following theorem is our first main result. 


\begin{thm}[Corollary \ref{min}]
With the above notation, let $\mu$ be the minimal face of $\sigma$ containing $w$ and  let $\tau\in\Delta$ be the minimal cone containing $w$. Then, 
\[
\dim \tau + \dim (\mu^\perp +V_w) =d.
\]
Equivalently, if $E_w$ denotes the divisor over $X$ corresponding to $w$, then the center of $E_w$ on $\widetilde{X}$ has dimension 
$\dim (\mu^\perp +V_w) $.
\end{thm}

In particular, if there are $d-1$ linearly independent $w$-critical arrows, then $E_w$ appears on $\widetilde{X}$ as a prime divisor. If there is at least one $w$-cirtical arrow, then the center of $E_w$ on $\widetilde{X}$ has positive dimension, and $E_w$ does not contract to a point on $\widetilde{X}$.
We give an algorithm for computing the dimension of $\mu^\perp +V_w$ and apply it to get a sufficient condition for the existence of $w$-critical arrows (Theorem \ref{polytope}).  Using this condition, we obtain Theorems \ref{3-dim-term-intro}, \ref{3-can-div-intro} and \ref{example-intro}  below as partial answers to Problem \ref{FB-ess}:

\begin{thm}[Theorem \ref{3-dim-terminal}]\label{3-dim-term-intro}
Let $X$ be a three-dimensional $\Q$-factorial terminal affine toric variety. Then every essential divisor over $X$ appears on $\widetilde{X}$ as a prime divisor.
\end{thm}

There is also a different proof of this theorem that is based on K\k{e}dzierski's result about the $G$-Hilbert scheme \cite{Ked11} (see Remark \ref{Kedzierski}). Note also that this result is  not covered by the result mentioined above in \cite{CMDY25}, since this class of singularities does not admit a moderate toric resolution \cite[Corollary 3.7]{KY26} as already mentioned. 
For a slightly larger class of singularities, we obtain the following weaker result:

\begin{thm}[Corollary \ref{3-can-div}]\label{3-can-div-intro}
Let $X=\A^3/G$ be a three-dimensional canonical abelian quotient singularity. Then every essential divisor over $X$ has center of positive dimension  on $\widetilde{X}$.
\end{thm}

As far as we examined a large number of examples in the situation of the last theorem, every essential divisor actually appears as a prime divisor  on $\widetilde{X}$ (see Theorem \ref{many-example}). However, the authors do not know how to prove it in general. 
On the other hand, we also obtain negative answers to Problem \ref{FB-ess}. By numerical experiments using the mentioned algorithm, we show:
\begin{thm}[Example \ref{example}]\label{example-intro}
There exist $\Q$-factorial toric singularities $X$ of dimension four such that some essential divisor over $X$ has center of codimension two on  $\widetilde{X}$.
\end{thm}

The final main result is that a simple inequality implies the existence of at least one $w$-critical arrow. Recall that the discrepancy $a(E)$ of a divisor $E$ over $X$ is the coefficient of $E$ in the relative canonical divisor $K_{Y/X}$ associated to a proper birational morphism $Y\to X$ where $E$ appears on $Y$ as a prime divisor. In our setting, $a(E)>-1$, that is, $X$ has only log terminal singularities. 
Let $|G|$ be the order of the finite abelian group $G$ in the description of $X$ as a quotient variety $\A^d/G$.

%
%

\begin{thm}
Let $w\in \sigma^\circ \cap N$ with. If the inequality 
\[ 
H_{d,|G|}:=\frac{d}{(|G|(d-1)!)^{1/d}} >1
\] 
holds, then every toric  divisor $E$ over $X$ with $a(E)\leq0$ has center of positive dimension on $\widetilde{X}$. Equivalently, if $w\in \sigma^\circ \cap N$ satisfies $a(E_w)\leq0$, then there exists at least one $w$-critical arrow.
\end{thm}

For example, the above inequality holds if $|G| \le 13$ in dimension $3$, $|G|\le 42$ in dimension  4, and $|G| \le 130$ in dimension 5.

We provide a code for SageMath \cite{sagemath} realizing the algorithm mentioned above, which is described Section \ref{pseudo_code}. Examples  \ref{basic-exam} and \ref{example} are based on computations using this code. We also use a computer to prove Theorem \ref{3-can}, Corollary \ref{3-can-div} and Theorem \ref{many-example}. We put files including the above code and these computations at the arXiv webpage of this paper.

The paper is organized as follows. In Section \ref{prelim}, we set up notation and recall basic notions and results. 
In Section \ref{dim-crit}, we give a formula for the dimension of the minimal cone containing the given element $w\in \sigma \cap N\setminus\{0\}$. In Section \ref{compute-crit-arrow}, we describe an algorithm to determine this dimension. In Section \ref{dim-3-4}, we give results on three-dimensional terminal singularities and canonical singularities as well as give negative examples in dimension four. In last Section \ref{abelian}, we show that a simple inequality shows the existence of at least one $w$-critical arrow.


\section*{Acknowledgments}
We thank Daniel Duarte, Kohei Sato, Yusuke Sato and Yudai Yamamoto for helpful discussion. The first named author was supported by the SECIHTI project CF-2023-G-33. The second named author was supported by JST SPRING, Grant Number JPMJSP2138. The third named author was supported by JSPS KAKENHI Grant Numbers JP21H04994, JP23K25767 and JP24K00519.

\section{Preliminaries}\label{prelim}

\subsection{Toric varieties}\label{toric}

Let $M$ and $N$ be free abelian groups of rank $d$ which are dual to each other and let $M_\R\coloneqq M\otimes\R$ and $N_\R\coloneqq N\otimes\R$. The natural pairing of $a\in M_\R$ and $w\in N_\R$ is denoted by $(a,w)$. We sometimes denote it by $w(a)$, regarding $w$ as a linear function on $M_\R$.

Fixing a basis of $M$, we sometimes identify $M$ with $\Z^d$. The dual basis then induces an identification $N=\Z^d$. Through these identifications, the vector spaces $M_\R$ and $N_\R$ are given with the standard inner products. In turn, the inner products enable us to talk about norms of vectors in these vector spaces as well as orthogonality among vectors in the same vector space.

For a subset $W\subset N_\R$, we define its outer orthogonal subspace $W^\perp\subset M_\R$ by
\begin{equation*}
W^\perp\coloneqq\{a\in M_\R\mid\forall w\in W,(a,w)=0\}.
\end{equation*}
Similarly, we define the outer orthogonal subspaces of a subset of $M_\R$.

Let $\sigma\subset N_\R$ be a strongly convex, nondegenerate (that is $d$-dimensional), rational polyhedral cone and let $\sigma^\vee\subset M_\R$ be its dual cone. Here, \emph{strongly convex} means that the cone contains the origin as a face. Following \cite{Yas12}, we denote $\sigma^\vee$ by $A_\R$ and define the monoid $A\coloneqq A_\R\cap M$. The affine toric variety $X=X_\sigma$ associated to $\sigma$ is given by $X=X_\sigma=\Spec k[A]$.

By a \emph{toric divisor over $X$}, we mean a toric prime divisor $E$ on $Z$ for some toric birational morphism $Z\to X$ of normal toric varieties. We identify two toric divisors over $X$, $E\subset Z$ and $E'\subset Z'$, if $E$ and $E'$ map to each other by the natural birational map between $Z$ and $Z'$. We say that a toric divisor $E$ over $X$ is \emph{exceptional} if the image of $E$ in $X$ has codimension $>1$.

A \emph{ray} in $N_\R$ means a one-dimensional cone of the form $\R_{\geq0}w$ for a nonzero element $w\in N$. A \emph{ray} of $\sigma $ means a one-dimensional face of $\sigma$, while a ray \emph{in} $\sigma$ is a ray contained in $\sigma$. A \emph{primitive element} of $N$ means a nonzero element which is not divisible by any integer $n\geq2$. Each ray contains a unique primitive element. As is well-known, there are one-to-one correspondences
\begin{align*}
\{\text{rays  in $\sigma$}\}&\leftrightarrow\{\text{primitive elements in $\sigma\cap N$}\} \\
&\leftrightarrow\{\text{toric divisors over $X$}\}
\end{align*}
By these correspondences, rays of $\sigma$ correspond to toric prime divisors on $X$.

We also consider the monoid $B\coloneqq\sigma\cap N$. We define a partial order on $B$ as follows; for $w,w'\in B$, $w\leq w'$ if and only if $w'\in w+\sigma$. As is well-known, the monoid $B$ has a unique minimal set of generators, which is called the \emph{Hilbert basis} of $B$, and we denote it by $\Hilb_N(\sigma)$. We write $w<w'$ if $w\leq w'$ and $w\neq w'$.

\subsection{F-blowups}

For each positive integer $l$, we put
\begin{align*}
(1/l)\cdot M&\coloneqq\{m/l\in M_\R\mid m\in M\}, \\
(1/l)\cdot A&\coloneqq A_\R\cap(1/l)\cdot M.
\end{align*}
We define $X_{(l)}\coloneqq\Spec k[(1/l)\cdot A]$, which is an affine toric variety isomorphic to $X$. The inclusion $A\hookrightarrow(1/l)\cdot A$ of monoids define a toric morphism
\begin{equation*}
F^{(l)}\colon X_{(l)}\to X.
\end{equation*}
If $k$ has characteristic $p>0$ and if $l=p^e$ for a positive integer $e$, then $F^{(l)}$ is nothing but the $e$-iterate Frobenius morphism of $X$.
\begin{dfn}[\cite{Yas12}]
We define the \emph{$(l)$-th F-blowup} of $X$, denoted by $\FB_{(l)}(X)$, to be the universal birational flattening of $F^{(l)}$.
\end{dfn}

By definition, there exists a natural proper birational morphism $\FB_{(l)}(X)\to X$. Thus, the normalization $\widetilde{\FB}_{(l)}(X)$ of $\FB_{(l)}(X)$ is a normal toric variety and obtained by a subdivision of the cone $\sigma$. From \cite[Theorem 3.12]{Yas12}, the sequence of F-blowups,
\begin{equation*}
X=\FB_{(1)}(X),\FB_{(2)}(X),\FB_{(3)}(X),\ldots,
\end{equation*}
stabilizes.
\begin{dfn}
We define $\FB_{(\infty)}(X)$ to be $\FB_{(l)}(X)$ for sufficiently large $l$ and call it the \emph{limit F-blowup of $X$}. We define the \emph{normalized limit F-blowup} $\widetilde{X}=\widetilde{\FB}_{(\infty)}(X)$ to be its normalization. We denote by $\Delta$ the fan corresponding to the toric proper birational morphism $\widetilde{X}\to X$, which is a subdivision of $\sigma$.
\end{dfn}

\begin{rem}
From \cite[Proposition 4.2.7]{CLS11}, a cone $\sigma$ is simplicial if and only if the associated toric variety $X$ is $\Q$-factorial. This is also equivalent to $X$ is an abelian finite quotient singularity \cite[Theorem 11.4.8]{CLS11}. Thus, in this case, $\widetilde{X}$ is isomorphic to the $G$-Hilbert scheme \cite{Yas12,TY09}.
\end{rem}

\subsection{Critical arrows}\label{critical-arrows}

Let $a_1,\ldots,a_n$ be the Hilbert basis of the monoid $A$ and let
\begin{equation}\label{Theta}
\Theta\coloneqq\bigcup_i(A_\R+a_i)\subset A_\R.
\end{equation}
Then,  $A_\R\setminus\Theta$ is a bounded set. We define  $D\in\R_{>0}$ to be the diameter of $A_\R\setminus\Theta$ (with respect to the chosen identification $M_\R=\R^d$).  Namely,
\begin{equation}\label{diameter}
D\coloneqq\sup\{|x-y|\mid x,y\in A_\R\setminus\Theta\}.
\end{equation}
We define a finite set $M_{\leq D}\coloneqq\{m\in M\mid|m|\leq D\}$.
We define $\Delta^\ast$ to be the fan obtained by subdividing $\sigma$ by hyperplanes $m^\perp\subset\R^d$ associated to the elements  $m\in M$ with $|m|\le D$. This fan is a refinement of $\Delta$ \cite[Lemma 2.10]{CMDY25}.
We say that a vector $w\in N_\R$ is \emph{general} if for every $a\in M_{\leq D}$, $w(a)\neq0$.

\begin{dfn}[{\cite[Definition 5.1]{CMDY25}}]
By an \emph{arrow}, we mean an ordered pair $\alpha=(\head_\alpha,\tail_\alpha)$ of distinct points $\head_\alpha,\tail_\alpha\in M_\R$, which we call the \emph{head} and \emph{tail}, respectively. The \emph{vector associated to} an arrow $\alpha$ is defined to be $\vect_\alpha\coloneqq\head_\alpha-\tail_\alpha$. We say that an arrow $\alpha$ is \emph{integral} if $\vect_\alpha\in M$ and \emph{primitive} if $\vect_\alpha$ is a primitive element of $M$.
\end{dfn}

Let $P,Q\in M_\R$ and let $\alpha$ be an integral arrow. We say that $P$ is \emph{$Q$-integral} if $P-Q\in M$. We say that $P$ is \emph{$\alpha$-integral} if $P-\head_\alpha\in M$ or equivalently if $P-\tail_\alpha\in M$.

The fixed identification $M=\Z^d$ induces an identification $N=\Z^d$ and hence the standard Euclidean metric on $N_\R$. 
Let $w\in\sigma\setminus\{0\}$ and let $w^{(\perp)}\subset N_\R$ be its inner orthogonal space. A \emph{tie-breaker} of $w$ is an element $b\in w^{(\perp)}$ such that $w+b$ is in a chamber of $\Delta^\ast$ whose closure contains $w$. Here, a \emph{chamber} of a fan to be the interior of a nondegenerate cone belonging to the fan.

Before defining $w$-critical arrows, we introduce the following notation. For basic properties, see \cite[Section 3]{CMDY25}. For $c\in\R_{>0}$, we define
\begin{align*}
\Lambda_{(=c)}^w&\coloneqq A_\R\cap\{w=c\}, \\
\Lambda_{(\leq c)}^w&\coloneqq A_\R\cap\{w\leq c\}, \\
\Lambda_{(<c)}^w&\coloneqq A_\R\cap\{w<c\}.
\end{align*}
Here, $\{w=c\}$ (resp. $\{w\leq c\}$, $\{w<c\}$) means $\{a\in M_\R\mid (a,w)=c\}$ (resp. $\{a\in M_\R\mid (a,w)\leq c\}$, $\{a\in M_\R\mid (a,w)<c\}$). We call $\Lambda_{(=c)}^w$ and $\Lambda_{(\leq c)}^w$ the \emph{level $c$ subpolyhedron} and the \emph{level $\leq c$ subpolyhedron} of $A_\R$, respectively. When $w$ is clear from the context, we simply write $\Lambda_{(=c)},\Lambda_{(\leq c)}$ and $\Lambda_{(<c)}$.

\begin{dfn}[{\cite[Definition 5.4]{CMDY25}}]\label{crit-arrow}
Let $w\in\sigma\setminus\{0\}$ and let $b$ be a tie-breaker of $w$. We say that an integral arrow $\alpha$ in $A_\R$ is \emph{$(w,b)$-critical} if for some $c\in\R_{>0}$,
\begin{enumerate}
\item $\head_\alpha,\tail_\alpha\in\Lambda_{(=c)}$,
\item there is no $\alpha$-integral point in $\Lambda_{(<c)}$,
\item $\tail_\alpha$ is the unique point where the restriction $b|_{\Lambda_{(=c)}}$ of the linear function $b\colon M_\R\to\R$ takes the minimum value.
\end{enumerate}
We say that an arrow is \emph{$w$-critical} if it is $(w,b)$-critical for some tie-breaker $b$ of $w$.
\end{dfn}
Note that from the first condition above, a $w$-critical arrow has the associated vector $\vect_\alpha$ orthogonal to $w$.

\begin{lem}[{\cite[Lemma 2.11]{CMDY25}}]\label{Xiw}
Let $w\in\sigma$ be a general element. Let $>$ be a total order on $\Z^d$ refining the partial order $\geq_w$. Let $M_{\leq D}^{w+}\coloneqq\{m\in M_{\leq D}\mid w(m)>0\}$. The following subsets of $A_\R$ are identical:
\begin{enumerate}
\item $\Theta\cup\{a\in A_\R\mid\exists b\in A_\R,a-b\in M_{\leq D}^{w+}\}$.
\item $\{a\in A_\R\mid\exists b\in A_\R,a-b\in M\text{ and } w(a)>w(b)\}$.
\item $\{a\in A_\R\mid\exists b\in A_\R,a-b\in M\text{ and } a>b\}$.
\end{enumerate}
\end{lem}
\begin{dfn}
For a general $w\in\sigma$, we denote by $\Xi_w$ the subset of $A_\R$ defined in the above Lemma.
\end{dfn}

\begin{prop}[{\cite[Corollary 2.15]{CMDY25}}]\label{same-chamber}
Let $w_1,w_2\in\sigma$ be general elements. Then, $w_1$ and $w_2$ are in the same chamber of $\Delta$ if and only if $\Xi_{w_1}=\Xi_{w_2}$.
\end{prop}

\begin{prop}\label{exists-crit-arrow}
Let $w\in B\setminus\{0\}$ and let $w_1,w_2$ be general elements sufficiently close to $w$. If $w_1,w_2$ are in the different chambers of $\Delta$ whose closure contains $w$, then there exists a $w$-critical arrow $\alpha$ in $A_\R$ such that $w_1(\vect_\alpha)\neq w_2(\vect_\alpha)$.
\end{prop}

\begin{proof}
From Proposition \ref{same-chamber}, we have $\Xi_{w_1}\neq\Xi_{w_2}$. Let $\head\in\Xi_{w_1}\setminus\Xi_{w_2}$. Then there exists $\tail\in A_\R$ such that $\head-\tail\in M$ and $w_1(\head)>w_1(\tail)$. Since $\head\notin\Xi_{w_2}$ and $w_2$ is general, we have $w_2(\head)<w_2(\tail)$. Thus $\alpha=(\head,\tail)$ is an integral arrow such that $c\coloneqq w(\head)=w(\tail)$. Now we show that there is no $\alpha$-integral point in $\Lambda_{(<c)}$. Suppose that there exists an $\alpha$-integral point $P$ in $\Lambda_{(<c)}$. Then we have $w(\head)>w(P)$. Since $w_1$ and $w_2$ are sufficiently close to $w$, we have
\[
w_1(\head)>w_1(P),w_2(\head)>w_2(P).
\]
However, this contradicts the assumption that $\head\notin\Xi_{w_2}$, and hence $\alpha$ is a $w$-critical arrow.
\end{proof}

\section{A dimension formula}\label{dim-crit}

We keep the notation from Section \ref{prelim}. In particular, $\sigma$ is a cone in $N_\R$, $B$ is the monoid $\sigma\cap N$, and $A_\R\subset M_\R$ is the dual cone of $\sigma$. Let $w\in B\setminus\{0\}$ and let $\mu$ be the minimal face of $\sigma$ containing $w$.
\begin{dfn}\label{wcrit}
We define a subspace $V_w$ of $M_\R$ as follows:
\begin{equation*}
V_w\coloneqq\langle\vect_\alpha\mid\alpha\text{ is a $w$-critical arrow in }A_\R\rangle_\R.
\end{equation*}
\end{dfn}

\begin{lem}\label{dim-formula}
Let $(V_w)^\perp$ be the orthogonal space of $V_w$, let $\overline{M_\R}$ be the quotient space $M_\R/\mu^\perp$ and let $\overline{V_w}$ be the image of $V_w$ in the quotient space $\overline{M}_\R$. Then we have
\[
\dim(\langle\mu\rangle_\R\cap(V_w)^\perp)=\dim\mu-\dim\overline{V_w}.
\]
\end{lem}
\begin{proof}
We have the following equalities:
\begin{align*}
\dim(\langle\mu\rangle_\R+(V_w)^\perp)&=\dim\mu+\dim(V_w)^\perp-\dim(\langle\mu\rangle_\R\cap(V_w)^\perp), \\
\dim(\langle\mu\rangle_\R+(V_w)^\perp)&=\dim((\mu^\perp\cap V_w)^\perp)=d-\dim(\mu^\perp\cap V_w), \\
\dim V_w&=d-\dim(V_w)^\perp, \\
\dim\overline{V_w}&=\dim V_w-\dim(\mu^\perp\cap V_w).
\end{align*}
Using these equalities in turn, we obtain
\begin{align*}
\dim(\langle\mu\rangle_\R\cap(V_w)^\perp)&=\dim\mu+\dim(V_w)^\perp-\dim(\langle\mu\rangle_\R+(V_w)^\perp) \\
&=\dim\mu+\dim(V_w)^\perp-(d-\dim(\mu^\perp\cap V_w)) \\
&=\dim\mu-\dim V_w+\dim(\mu^\perp\cap V_w) \\
&=\dim\mu-\dim\overline{V_w}.
\end{align*}
\end{proof}

Let $\tau\in\Delta$ be the minimal cone containing $w$. Let us recall the following result from \cite{CMDY25} which results $\tau$ and $w$-critical arrows.

\begin{prop}[{\cite[Theorem 5.7]{CMDY25}}]\label{min-cone-inclusion}
We have $\tau \subset (V_w)^\perp$.
\end{prop}

We improve this result below.

\begin{thm}\label{arrow-dim}
We have $\mu^\perp+V_w$ and $\langle\tau\rangle_\R$ are the orthogonal subspaces of each other. In particular, we have $\dim \tau=\dim\mu-\dim\overline{V_w}$.
\end{thm}

\begin{proof}
From Proposition \ref{min-cone-inclusion}, we have $\langle\tau\rangle_\R\subset\langle\mu\rangle_\R\cap(V_w)^\perp$. Suppose that $\dim\tau<\dim(\langle\mu\rangle_\R\cap(V_w)^\perp)$ and 
let $u\in\mu\cap(V_w)^\perp\setminus\langle\tau\rangle_\R$. We take a sufficiently small $\delta$ so that $w\pm\delta u$ are in the relative interior $\mu^\circ$ of $\mu$.

We now recall the notion of star. For each cone $\gamma\in\Delta$ containing $\tau$, let $\widetilde{\gamma}$ be the image cone in $N_\R/\langle\tau\rangle_\R$ under the quotient map $N_\R\to N_\R/\langle\tau\rangle_\R$. Then
\begin{equation*}
\Star(\tau)\coloneqq\{\widetilde{\gamma}\subset N_\R/\langle\tau\rangle_\R\mid\tau\preceq\gamma\in\Delta\}
\end{equation*}
is a fan in $N_\R/\langle\tau\rangle_\R$, called the star of $\tau$ (see, for example \cite[Section 3.1]{Ful93}).

Since all cones of $\Star(\tau)$ have the $\widetilde{\tau}$ as a face, equivalently, all cones of $\Star(\tau)$ are strongly convex, no cone can contain both $\delta\widetilde{u}$ and $-\delta\widetilde{u}$. Thus, $\delta\widetilde{u}$ and $-\delta\widetilde{u}$ are not in the closure of same chamber of $\Star(\tau)$. Let $b$ be a tie-breaker of $w$ and let
\begin{equation*}
w_1\coloneqq w+\delta u+\varepsilon b,w_2\coloneqq w-\delta u+\varepsilon b,
\end{equation*}
where $\varepsilon>0$. Then, for sufficiently small $0<\varepsilon\ll\delta$, $\widetilde{w_1}$ and $\widetilde{w_2}$ are in different chambers of $\Star(\tau)$. Thus, $w_1$ and $w_2$ are in different chambers of $\Delta$ whose closure contains $\tau$ as a face. Hence, from Proposition \ref{exists-crit-arrow}, there exists a $w$-critical arrow $\alpha$ in $A_\R$ such that $w_1(\vect_\alpha)\neq w_2(\vect_\alpha)$. However, since $w_1-w_2=2\delta u\in(V_w)^\perp$, for every $w$-critical arrow $\beta$ in $A_\R$, we have $w_1(\vect_\beta)=w_2(\vect_\beta)$, which is a contradiction. Therefore, we obtain
\[
\langle\tau\rangle_\R=\langle\mu\rangle_\R\cap(V_w)^\perp=(\mu^\perp+V_w)^\perp.
\]
In particular, from Lemma \ref{dim-formula}, we have $\dim\tau=\dim\mu-\dim{\overline{V_w}}$.
\end{proof}


\begin{cor}\label{min}
The prime divisor associated with $w$ has center of codimension 
$\dim\mu-\dim\overline{V_w}$ on $\widetilde{X}$.
\end{cor}

\begin{proof}
Let $\Delta'$ be a refinement of $\Delta$ containing the ray $ \R_{\geq 0}w$ and let $X'$ be the corresponding toric variety. The prime divisor $E_w \subset X'$ corresponding to $\R_{\geq 0}w$ is the closure of the torus orbit $O_w$ corresponding $\R_{\geq 0}w$. From
\cite[Lemma 3.3.21(a)]{CLS11}, the image $O_w$ in $\widetilde{X}$ is the torus orbit $O_\tau$ corresponding to $\tau$. This shows that the center of $E_w$ on $\widetilde{X}$ is the closure of $O_\tau$ and has codimension equal to the dimension of $\tau$.  From the Theorem \ref{arrow-dim}, this is equal to $\dim \mu -\dim \overline{V_w}$. 
\end{proof}

\section{A computational method of finding critical arrows}\label{compute-crit-arrow}

In order to determine the dimension of the minimal cone $\tau\in\Delta$ containing $w$, we establish a computational method for finding the $w$-critical arrows.

\subsection{A description of the set of critical arrows with bounded lengths}

We fix $w\in B\setminus\{0\}$ and let $\mu$ be the minimal face of $\sigma$ containing $w$. Thus, $w$ is in the relative interior $\mu^\circ$ of $\mu$. The dual face $\mu^\ast$ of $\mu$ is, by definition, the face of $\sigma^\vee=A_\R$ given by $\sigma^\vee\cap\mu^\perp$. From \cite[Exercise 1.2.2]{CLS11}, we have $\mu^\ast=A_\R\cap w^\perp$. From $\dim\mu+\dim\mu^\ast=d$, we also deduce $\langle\mu^\ast\rangle_\R=\mu^\perp$.

\begin{prop}\label{Crit-vect}
For any $D'\geq D$ (see \eqref{diameter} in Section \ref{critical-arrows} for the definition of $D$), we have
\begin{equation*}
\overline{V_w}=\langle\overline{\vect_\alpha}\mid\alpha\text{ is a $w$-critical arrow in $A_\R$ with }|\vect_\alpha|\leq D'\rangle_\R.
\end{equation*}
\end{prop}

\begin{proof}
It suffices to prove the statement in case $D'=D$. Let $\alpha$ be a $(w,b)$-critical arrow. According to \cite[Lemma 5.5]{CMDY25}, we have $\tail_\alpha\notin\Xi_{w+\delta b}$ for all $\delta\in(0,1]$ (see Section \ref{critical-arrows}). Therefore, from Lemma \ref{Xiw}, we have $\tail_\alpha\notin\Theta$. Suppose that $\head_\alpha\in\Theta$. Then, from the definition of $\Theta$ (\eqref{Theta} in Section \ref{critical-arrows}), there exists $\head\in A_\R$ such that $\head_\alpha-\head=a$, where $a$ is an element of the Hilbert basis of $A_\R$. If $\head=\tail_\alpha$, then $|\vect_\alpha|=|a|\leq D$. If $\head\neq\tail_\alpha$, then by condition (2) in the Definition \ref{crit-arrow}, we have $a\in w^\perp$. In this case, $\alpha'=(\head,\tail_\alpha)$ is again a $w$-critical arrow. If $\head\in\Theta$, then we apply the above argument to $\alpha'$ to get a new $w$-critical arrow $\alpha''=(\head',\tail_\alpha)$, where $\head-\head'$ is an element of the Hilbert basis. Repeating this operation finitely many times, we eventually get a $w$-critical arrow $\beta=(\head_\beta,\tail_\alpha)$ with $\head_\beta\notin\Theta$. Then, we have $|\vect_\beta|\leq D$. Since $\head_\alpha-\head_\beta$ is contained in $A_\R\cap w^\perp\subset\mu^\perp$, we have $\overline{\vect_\alpha}=\overline{\vect_\beta}$ in $\overline{V_w}$.
\end{proof}

\begin{dfn}\label{delta}
Let $\sigma=\langle v_1,\ldots,v_d\rangle_{\R_{\geq0}}$ be a nondegenerate simplicial cone of $N_\R$. Let $w\in B\setminus\{0\}$ and let $\mu$ be the minimal face of $\sigma$ containing $w$. For each $1\leq i\leq d$ with $v_i\in\mu$, we define cones $\delta_i^w\subset N_\R$ by
\[
\delta_i^w=\delta_i\coloneqq\langle v_1,\ldots,v_{i-1},-w,v_{i+1},\ldots,v_d\rangle_{\R_{\geq0}}.
\]
\end{dfn}
By the choice of $i$, $\delta_i^w$ is a $d$-dimensional simplicial cone. Let $i\in\{1,\ldots,d\}$ such that $v_i\in\mu$ and let $v_1^\ast,\ldots,v_d^\ast$ be a dual basis of $v_1,\ldots,v_d$. Let
\[
L=\{l\in\{1,\ldots,d\}\setminus\{i\}\mid v_l^\ast\in w^\perp\},R=\{r\in\{1,\ldots,d\}\setminus\{i\}\mid v_r^\ast\notin w^\perp\},
\]
so that we have $\{1,\ldots,d\}=L\sqcup R\sqcup\{i\}$. We define vectors $u_1,\ldots,u_d$ by rescaling $v_1^\ast,\ldots,v_d^\ast$ as follows: For $j\in R\cup\{i\}$, we set $u_j$ to be the positive multiple of $v_j^\ast$ with $(u_j,w)=1$. For $l\in L$, we set $u_l=v_l^\ast$.
\begin{lem}\label{delta-dual}
We have
\[
(\delta_i^w)^\vee=\langle-u_i,u_l,u_r-u_i\mid l\in L,r\in R\rangle_{\R_{\geq0}}.
\]
\end{lem}

\begin{proof}
Note that $(-u_i,v_j)=0$ for all $j\in\{1,\ldots,d\}\setminus\{i\}$ and $(-u_i,-w)=1$. For $l\in L$, we have $(u_l,v_j)=0$ for all $j\in\{1,\ldots,d\}\setminus\{l,i\}$ and $(u_l,-w)=0$. For $r\in R$, we have $(u_r-u_i,v_j)=0$ for all $j\in\{1,\ldots,d\}\setminus\{r,i\}$ and $(u_r-u_i,-w)=0$. Let $c_l\coloneqq(u_l,v_l)\in\R_{>0}$ for $l\in L$ and $c_r\coloneqq(u_r-u_i,v_r)=(u_r,v_r)\in\R_{>0}$ for $r\in R$. Then, the vectors $-u_i,c_l^{-1}u_l,c_r^{-1}(u_r-u_i)$ are dual basis of $v_1,\ldots,v_{i-1},-w,v_{i+1},\ldots,v_d$. Therefore, we have
\begin{align*}
(\delta_i^w)^\vee&=\langle -u_i,c_l^{-1}u_l,c_r^{-1}(u_r-u_i)\mid l\in L,r\in R\rangle_{\R_{\geq0}} \\
&=\langle-u_i,u_l,u_r-u_i\mid l\in L,r\in R\rangle_{\R_{\geq0}}.
\end{align*}
\end{proof}

Let $i\in\{1,\ldots,d\}$ such that $v_i\in\mu$. We define a linear function
\begin{equation}\label{linear-function}
f_i\colon\R\to\R;x\mapsto x(u_i,v_i).
\end{equation}
Let $\tau_i$ denote the unique facet of $A_\R$ which does not contain $u_i$, namely
\begin{equation}\label{facet}
\tau_i=\langle u_l,u_r\mid l\in L,r\in R\rangle_{\R_{\geq0}}.
\end{equation}

\begin{lem}\label{transpose}
For $c\in\R_{>0}$, we have
\[
\Lambda_{(\leq c)}^w=((\delta_i^w)^\vee\cap\{-v_i\leq f_i(c)\})+cu_i.
\]
\end{lem}

\begin{proof}
With the notation introduced in Lemma \ref{delta-dual} the dual cone $A_\R$ of $\sigma$ can be written as
\begin{align*}
A_\R&=\langle v_1^\ast,\ldots,v_d^\ast\rangle_{\R_{\geq0}} \\
&=\langle u_i,u_l,u_r\mid l\in L,r\in R\rangle_{\R_{\geq0}}.
\end{align*}
Let $P\in\Lambda_{(\leq c)}$. Since $\Lambda_{(\leq c)}\subset A_\R$, we may write
\[
P=p_iu_i+\sum_{l\in L}p_lu_l+\sum_{r\in R}p_ru_r
\]
with $p_i,p_l,p_r\in\R_{\geq0}$. Since $(P,w)=p_i+\sum_{r\in R}p_r\leq c$, Lemma \ref{delta-dual} gives
\[
P-cu_i=\left((p_i+\sum_{r\in R}p_r)-c\right)u_i+\sum_{l\in L}p_lu_l+\sum_{r\in R}p_r(u_r-u_i)\in(\delta_i^w)^\vee.
\]
On the other hand, we can also write
\[
P-cu_i=(p_i-c)u_i+\sum_{l\in L}p_lu_l+\sum_{r\in R}p_ru_r,
\]
which shows
\[
(P-cu_i,-v_i)=(c-p_i)(u_i,v_i)\leq c(u_i,v_i)=f_i(c).
\]
Thus, $P\in((\delta_i^w)^\vee\cap\{-v_i\leq f_i(c)\})+cu_i$ and hence $\Lambda_{(\leq c)}\subset((\delta_i^w)^\vee\cap\{-v_i\leq f_i(c)\})+cu_i$.

To show the opposite, let $u\in(\delta_i^w)^\vee\cap\{-v_i\leq f_i(c)\}$. Then from Lemma \ref{delta-dual}, we may write
\[
u=p_i(-u_i)+\sum_{l\in L}p_lu_l+\sum_{r\in R}p_r(u_r-u_i)
\]
with $p_i,p_l,p_r\in\R_{\geq0}$. Since
\[
\left(p_i+\sum_{r\in R}p_r\right)(u_i,v_i)=(u,-v_i)\leq f_i(c)=c(u_i,v_i),
\]
we have $\left(p_i+\sum_{r\in R}p_r\right)\leq c$. Thus, we have
\[
u+cu_i=\left(c-(p_i+\sum_{r\in R}p_r)\right)u_i+\sum_{l\in L}p_lu_l+\sum_{r\in R}p_ru_r\in A_\R.
\]
By definition of $\delta_i^w$ (Definition \ref{delta}), we have $(u,w)\leq0$. Hence, we have
\begin{equation}\label{u+cui,w}
(u+cu_i,w)=(u,w)+c(u_i,w)\leq c.
\end{equation}
Thus, we have $u+cu_i\in\Lambda_{(\leq c)}$. Therefore, we obtain
\[
\Lambda_{(\leq c)}\supset((\delta_i^w)^\vee\cap\{-v_i\leq f_i(c)\})+cu_i.
\]
\end{proof}

\begin{lem}\label{delta-Lambda}
For $c\in\R_{>0}$, the map
\begin{equation}\label{area}
\{\alpha=(\head_\alpha,\tail_\alpha)\in(\Lambda_{(\leq c)}^w)^2\mid\head_\alpha\in\tau_i,\tail_\alpha=cu_i,w(\head_\alpha)\leq c\}\to(\delta_i^w)^\vee\cap\{-v_i=f_i(c)\}
\end{equation}
defined by $\alpha\mapsto\vect_\alpha$ is a bijection. The inverse map is given by
\begin{equation}\label{area-area}
u\mapsto (u+cu_i,cu_i).
\end{equation}
Under this correspondence, an arrow $\alpha$ in $\Lambda_{(=c)}^w$ such that $\head_\alpha\in\tau_i,\tail_\alpha=cu_i$ corresponds to the associated vector $\vect_\alpha$ in $(\delta_i^w)^\vee\cap\{-v_i=f_i(c)\}\cap w^\perp$.
\end{lem}

\begin{proof}
Let $\alpha=(\head_\alpha,\tail_\alpha)$ be an arrow in $\Lambda_{(\leq c)}^w$ such that $\head_\alpha\in\tau_i$, $\tail_\alpha=cu_i$ and $w(\head_\alpha)\leq c$. From Lemma \ref{transpose}, we have $\vect_\alpha=\head_\alpha-\tail_\alpha\in(\delta_i^w)^\vee\cap\{-v_i\leq f_i(c)\}$. From \eqref{facet}, we have
\[
(\vect_\alpha,-v_i)=(\tail_\alpha,v_i)=c(u_i,v_i)=f_i(c),
\]
and hence, we have $\vect_\alpha\in(\delta_i^w)^\vee\cap\{-v_i=f_i(c)\}$. In particular, if the arrow $\alpha$ is in $\Lambda_{(=c)}^w$, then we have $(\vect_\alpha,w)=0$. Thus, $\vect_\alpha\in(\delta_i^w)^\vee\cap\{-v_i=f_i(c)\}\cap w^\perp$.

Let $u\in(\delta_i^w)^\vee\cap\{-v_i=f_i(c)\}$. Then from Lemma \ref{transpose}, $u+cu_i\in\Lambda_{(\leq c)}$. Then from \eqref{linear-function} (Definition of $f_i$), we have
\[
(u+cu_i,-v_i)=(u,-v_i)+(cu_i,-v_i)=f_i(c)-c(u_i,v_i)=0.
\]
From \eqref{facet}, we have $\tau_i=A_\R\cap\{-v_i=0\}$ and therefore $u+cu_i\in\tau_i$. By the same computation as in \eqref{u+cui,w}, it follows that $w(u+cu_i)\leq c$. Thus, assignment \eqref{area-area} defines a map in the direction opposite to \eqref{area}. It is clear that this map is inverse map of \eqref{area}. Therefore, \eqref{area} is a bijection. 
In particular, if $u\in(\delta_i^w)^\vee\cap\{-v_i=f_i(c)\}\cap w^\perp$, then
\[
(u+cu_i,w)=c(u_i,w)=c.
\]
Hence, $\alpha=(u+cu_i,cu_i)$ is an arrow in $\Lambda_{(=c)}^w$.
\end{proof}
Figures 1 and 2 illustrate the correspondence given in Lemmas \ref{transpose} and \ref{delta-Lambda}.
\begin{figure}[h]
\centering
\begin{minipage}{0.45\textwidth}
\centering
\begin{tikzpicture}
\draw (-0.2,-0.2) node{$O$};
\fill[blue!15, opacity=0.5] (0,0) -- (4,0.5) -- (2,3) -- cycle;
\draw[red,line width=1.5pt] (4,0.5) -- (2,3);
\draw[->] (0,0) node at (6,1){$u_1$}--(6,0.75);
\draw[->] (0,0) node at (2.7,4.8){$u_2$}--(3,4.5);
\draw (1.9,1) node{$\Lambda_{(\leq c)}$};
\draw (2.6,2.9) node{$\Lambda_{(=c)}$};
\draw[orange!90!black, very thick, ->] (4,0.5) node at (2.2,1.5){$\alpha$}--(1.5,2.25);
\fill[orange!90!black] (4,0.5) circle (2pt) node[above right]{$\tail_\alpha$};
\fill[orange!90!black] (1.5,2.25) circle (2pt) node[above left]{$\head_\alpha$};
\end{tikzpicture}
\caption{$\sigma^\vee$}
\end{minipage}
\hfill
\begin{minipage}{0.45\textwidth}
\centering
\begin{tikzpicture}
\draw (4.2,0.3) node{$O$};
\draw[->] (4,0.5) node at (-1.9,0){$-u_1$}--(-2,-0.25);
\draw[->] (4,0.5) node at (0.5,4){$u_2-u_1$}--(1,4.25);
\draw[red,line width=1.5pt] (4,0.5) --(2,3);
\fill[blue!15, opacity=0.5] (0,0) -- (4,0.5) -- (2,3) -- cycle;
\draw (1.9,1) node{$\{-v_1\leq f(c)\}$};
\draw (3.7,2.9) node{$\{-v_1\leq f(c)\}\cap w^\perp$};
\draw[orange!90!black, very thick, ->] (4,0.5) node at (2.2,1.5){$\vect_\alpha$}--(1.5,2.25);
\end{tikzpicture}
\caption{$(\delta_1^w)^\vee$}
\end{minipage}
\end{figure}

\begin{dfn}\label{Crit-delta}
Let
\[
c_{i,w}=c\coloneqq\min\{(u,-v_i)\mid u\in(\delta_i^w)^\vee\cap M\setminus w^\perp\}\in\R_{>0}.
\]
Then, we define sets $\Crit_i^w\subset M$ as
\[
\Crit_i^w\coloneqq\{u\in(\delta_i^w)^\vee\cap M\setminus\{0\}\mid(u,-v_i)<c_{i,w}\}\subset w^\perp.
\]
We also define
\[
\Crit^w\coloneqq\bigcup_{\substack{1\leq i\leq d \\v_i\in\mu}}\Crit_i^w.
\]
\end{dfn}

\begin{prop}\label{Crit-set}
Every element of $\Crit_i^w$ 
is the vector associated to some $w$-critical arrow in $A_\R$. Conversely, every $w$-critical arrow with tail on the ray $\R_{\geq0}u_i$ has the associated vector in $\Crit_i^w$.
\end{prop}

\begin{proof}
With the notation introduced in Lemma \ref{delta-dual}, the dual cone $A_\R$ of $\sigma$ can be written as
\begin{align*}
A_\R&=\langle v_1^\ast,\ldots,v_d^\ast\rangle_{\R_{\geq0}} \\
&=\langle u_i,u_l,u_r\mid l\in L,r\in R\rangle_{\R_{\geq0}}.
\end{align*}
From Lemma \ref{delta-Lambda}, each $u\in\delta_i^\vee\cap M\setminus\{0\}$ can be expressed as the vector associated to an arrow $\alpha=(\head_\alpha,\tail_\alpha)$ in $A_\R$ such that $\head_\alpha\in\tau_i$, $\tail_\alpha=cu_i$ for some $c\in\R_{>0}$ and $w(\head_\alpha)\leq c$. 
Suppose that there exists $u\in\Crit_i^w$ whose associated arrow $\alpha=(\head_\alpha,\tail_\alpha)$ in $\Lambda_{(=c)}^w$ is not a $w$-critical arrow. It was showed in the proof of \cite[Lemma 6.2]{CMDY25} that there exists a tie-breaker of $w$ satisfying condition (3) in Definition \ref{crit-arrow}. Hence, there exists $P\in A_\R$ such that $P-\tail_\alpha\in M$ and
\begin{equation}\label{taQ}
c=(\head_\alpha,w)=(\tail_\alpha,w)>(P,w).
\end{equation}
Since $P\in A_\R$, we may write
\begin{equation}\label{equation-u-Q}
P=p_iu_i+\sum_{l\in L}p_lu_l+\sum_{r\in R}p_ru_r
\end{equation}
with $p_i,p_l,p_r\in\R_{\geq0}$. 
From Lemma \ref{transpose}, we have $P-\tail_\alpha\in\delta_i^\vee\cap\{-v_i\leq f_i(c)\}\cap M$. From \eqref{equation-u-Q}, 
we see that
\begin{equation*}
(P-\tail_\alpha,-v_i)=(P-cu_i,-v_i)=(c-p_i)(u_i,v_i)\leq c(u_i,v_i).
\end{equation*}
From Lemma \ref{delta-Lambda} and by the choice of $u$, we have
\[
(u,-v_i)=f_i(c)=c(u_i,v_i)<c_{i,w}.
\]
Thus, $(P-\tail_\alpha,-v_i)\leq(u,-v_i)<c_{i,w}$, and by Definition \ref{Crit-delta}, we have $P-\tail_\alpha\in\Crit_i^w\subset w^\perp$. However, this contradicts \eqref{taQ}. Hence, $\alpha$ is a $w$-critical arrow.

Let $\alpha=(\head_\alpha,\tail_\alpha)$ be a $w$-critical arrow in $\Lambda_{(=c)}$ whose tail is contained in the ray $\R_{\geq0}u_i$. By Lemma \ref{transpose}, the associated vector $\vect_\alpha=\head_\alpha-cu_i$ is contained in $(\delta_i^w)^\vee\cap M\cap\{-v_i\leq f_i(c)\}\cap w^\perp\setminus\{0\}$. Suppose that $(\vect_\alpha,-v_i)\geq c_{i,w}$ and let $P$ be a lattice point in $(\delta_i^w)^\vee\cap M\setminus w^\perp$ such that $(P,-v_i)=c_{i,w}$. Then, by Lemma \ref{transpose}, $P+\tail_\alpha\in\Lambda_{(\leq f_i^{-1}(c_{i,w}))}\subset A_\R$. Since $w(P)<0$, we have
\[
w(P+\tail_\alpha)=w(P)+w(\tail_\alpha)<w(\tail_\alpha)=c.
\]
Hence, the arrow $\alpha$ does not satisfy condition (2) in Definition \ref{crit-arrow}, which contradicts the assumption that $\alpha$ is a $w$-critical arrow. Therefore, $(\vect_\alpha,-v_i)<c_{i,w}$, namely $\vect_\alpha\in\Crit_i^w$.
\end{proof}

\begin{cor}\label{Critw-set}
We have
\[
\Crit^w=\{\vect_\alpha\mid\alpha\text{ is a $w$-critical arrow in }A_\R\}.
\]
\end{cor}

\begin{proof}
Let $\alpha=(\head_\alpha,\tail_\alpha)$ be a $(w,b)$-critical arrow. By the definition of a $(w,b)$-critical arrow (Definition \ref{crit-arrow}), $\tail_\alpha$ is contained in one of the rays of $A_\R$. Thus, all $w$-critical arrows can be obtained by the construction given in the Proposition \ref{Crit-set}. On the other hand, by Definition \ref{Crit-delta}, we have
\[
\Crit^w=\bigcup_{\substack{1\leq i\leq d \\v_i\in\mu}}\Crit_i^w.
\]
Therefore, the reverse inclusion follows again from the Proposition \ref{Crit-set}.
\end{proof}

The following proposition is used to compute $\Crit^w$.
\begin{prop}\label{Crit-Hilb}
We keep the notation of the above proposition. Then at least one element of $u\in\delta_i^\vee\cap M\setminus w^\perp$ such that $(u,-v_i)=c_{i,w}$ belongs to the Hilbert basis of $\delta_i^\vee\cap M$.
\end{prop}

\begin{proof}
Let $u\in\delta_i^\vee\cap M\setminus w^\perp$ such that $(u,-v_i)=c_{i,w}$. Suppose that $u$ is not an element of the Hilbert basis of $\delta_i^\vee\cap M$. Then there exists $u_1\in \delta_i^\vee\cap M\setminus\{0\}$ such that $u_1<u$ with respect to the partial order defined in Section \ref{toric}. If $u_1,u-u_1\in w^\perp$, then $u=u_1+(u-u_1)\in w^\perp$, which contradicts the choice of $u$. Hence, replacing $u_1$ with $u-u_1$ if necessary, we may assume that $u_1\in\delta_i^\vee\cap M\setminus w^\perp$, $u_1<u$. 
If $u-u_1\notin(-v_i)^\perp$, then we have $(u-u_1,-v_i)>0$, and hence $(u_1,-v_i)<(u,-v_i)=c_{i,w}$. Therefore, by Definition \ref{Crit-delta}, we have $u_1\in\Crit_i^w\subset w^\perp$. However, this contradicts the choice of $u_1$. Thus, $u-u_1\in(-v_i)^\perp$, and hence
\[
(u,-v_i)=(u_1,-v_i)+(u-u_1,-v_i)=(u_1,-v_i).
\]
Hence, $u_1$ satisfies the same assumption as $u$. If $u_1$ is not an element of the Hilbert basis of $\delta_i^\vee\cap M$, then the above argument shows that there exists $u_2\in\delta_i^\vee\cap M\setminus w^\perp$ such that $u_2<u_1$. Repeating this operation, we obtain a strictly descending sequence $u_1>u_2>\cdots$ in $\delta_i^\vee\cap M\setminus w^\perp$ as long as $u_n$ is not an element of the Hilbert basis of $\delta_i^\vee\cap M$. By construction, each $u_n$ satisfies the same assumption as $u$. Since this partial order satisfies the descending chain condition, the process must terminate. Thus, there exists $n\in\Z_{>0}$ such that $u_n$ is an element of the Hilbert basis of $\delta_i^\vee\cap M$.
\end{proof}

For $D'\geq D$, we define sets $\Crit_{i,D'}^w\subset M$ as
\[
\Crit_{i,D'}^w\coloneqq\{u\in\Crit_i^w\mid|u|\leq D'\}\subset w^\perp.
\]
We also define
\[
\Crit_{D'}^w\coloneqq\bigcup_{\substack{1\leq i\leq d \\v_i\in\mu}}\Crit_{i,D'}^w.
\]
\begin{prop}
For $D'\geq D$, the set $\Crit_{D'}^w$ is a generating set of $V_w$.
\end{prop}

\begin{proof}
The assertion follows from Proposition \ref{Crit-vect} and Corollary \ref{Critw-set}.
\end{proof}

\subsection{Algorithm}\label{pseudo_code}

Based on results in the previous section, we can construct an algorithm for finding $w$-critical arrows with bounded lengths and computing the dimension of the minimal cone $\tau\in\Delta$ containing $w$.

\textbf{Input}: Generators $v_1,\ldots,v_d \in \Z^d=N$ of the simplicial cone $\sigma $ and $w\in B \setminus\{0\}$.

\textbf{Output}: The dimension of the minimal cone $\tau\in\Delta$ containing $w$.

\textbf{Procedure}: 
\begin{enumerate}
\item Compute a dual basis $v_1^\ast,\ldots,v_d^\ast$ of $v_1,\ldots,v_d$.
\item Compute $D'\coloneqq\lceil\max\{|u-u'|\mid u,u'\in S\}\rceil\geq D$, where $S=\{\sum a_iv_i^\ast\mid a_i\in\{0,1\}\}$ and $\lceil x\rceil$ denotes the smallest integer greater than or equal to $x$. (The elements of $S$ are the vertices of the parallelotope spanned $v_1^\ast,\ldots,v_d^\ast$ and $D'$ is the ceiling of its diameter. We have $D'\geq D$.)
\item Let $\mu$ be the minimal face of $\sigma$ containing $w$. For each $i$ with $v_i\in\mu$, compute a set $\Crit_{i,D'}^w$ of vectors as follows:
\begin{enumerate}
\item Let $\delta_i=\langle v_1,\ldots,v_{i-1},-w,v_{i+1},\ldots,v_d\rangle_{\R_{\geq0}}$.
\item Compute the Hilbert basis $H_i^w$ of the dual cone $\delta_i^\vee\cap\Z^d$.
\item Compute $c_{i,w}=\min\{(u,-v_i)\mid u\in H_i^w\setminus w^\perp\}$.
\item Compute $\Crit_{i,D'}^w=\{u\in\delta_i^\vee\cap M\setminus\{0\}\mid|u|\leq D',(u,-v_i)<c_{i,w}\}$.
\end{enumerate}
\item Put the set $\Crit_{D'}^w$ to be the union of all the $\Crit_{i,D'}^w$ computed in (3). (The set $\Crit_{D'}^w$ is a generating set of $V_w$).
\item Compute the dimension of the minimal cone $\tau$ containing $w$ by
\begin{equation*}
\dim\tau=\dim\mu-(\dim V_w-\dim(\mu^\perp\cap V_w)).
\end{equation*}
\end{enumerate}

\begin{rem}
The above algorithm enumerates all $w$-critical arrows in $A_\R$ whose length of associated vector is less than or equal to $D'$. Moreover, by a slight modification, the algorithm can be made to output a basis of $\langle\tau\rangle_\R$.
\end{rem}

\begin{exam}\label{basic-exam}
Let $\sigma$ be a cone generated by $v_1=(1,0,0)$, $v_2=(0,1,0)$, $v_3=(1,1,2)$. Note that $v_1^\ast=(2,0,-1)$, $v_2^\ast=(0,2,-1)$, $v_3^\ast=(0,0,1)$ and
\[
D'=\left\lceil|((2,0,-1)+(0,2,-1))-(0,0,1)|\right\rceil=\left\lceil\sqrt{17}\right\rceil=5.
\]
Set $w=(1,1,1)\in\sigma$, which is the unique element of the Hilbert basis of $B$ corresponding to the essential divisor over $X_\sigma$. Since
\[
w=\frac{1}{2}v_1+\frac{1}{2}v_2+\frac{1}{2}v_3,
\]
we have $\mu=\sigma$. Then for each $1\leq i\leq 3$, we have
\begin{align*}
\Crit_1^w=\{(-1,0,1),(-1,1,0)\}, \\
\Crit_2^w=\{(0,-1,1),(1,-1,0)\}, \\
\Crit_3^w=\{(0,1,-1),(1,0,-1)\}.
\end{align*}
Thus, $V_w=\langle\pm(1,-1,0),\pm(1,0,-1),\pm(0,1,-1)\rangle_\R$. Since $(1,0,-1)=(1,-1,0)+(0,1,-1)$, we have $\dim V_w=2$ and
\[
\dim\tau=\dim\mu-\dim V_w=3-2=1.
\]
Next, set $w=(1,2,2)\in\sigma$, which is not an element of the Hilbert basis. Since $w=v_2+v_3$, we have $\mu=\langle v_2,v_3\rangle_{\R_{\geq0}}$. Then for each $2\leq i\leq3$, we have
\[
\Crit_2^w=\Crit_3^w=\{(2,0,-1),(4,0,-2)\}.
\]
Thus, $V_w=\langle(2,0,-1),(4,0,-2)\rangle_\R$. Since $\mu^\perp=\langle(2,0,-1)\rangle_\R$, we have $\mu^\perp\cap V_w=V_w$. Hence, we obtain
\[
\dim\tau=\dim\mu-(\dim V_w-\dim(\mu^\perp\cap V_w))=2-(1-1)=2.
\]
\end{exam}

\section{Essential divisors of three-dimensional $\Q$-factorial toric singularities}\label{dim-3-4}

In this section, we study the relation between F-blowups and essential divisors for $\Q$-factorial toric singularities.

\subsection{Essential divisors and discrepancies}

We briefly review the notion of essential divisors.
\begin{dfn}[\cite{IK03,PS15}]
We say that a divisor over $X$ is \emph{essential} if, in any resolution, whose center is an irreducible component of the exceptional locus.
\end{dfn}

\begin{dfn}[\cite{BGS95}]
We say that an exceptional divisor over $X$ is \emph{BGS essential} if it appears on every resolution of $X$ as a prime divisor.
\end{dfn}

\begin{rem}[{\cite[Proposition A.8]{CMDY25}}]
If $X$ is a $\Q$-factorial toric variety, then a toric divisor $E$ over $X$ is essential if and only if it is BGS essential.
\end{rem}

\begin{prop}[{\cite[Theorem 1.10]{BGS95}}]
Let $\rho\subset\sigma$ be a ray and let $E_\rho$ be the corresponding exceptional 
divisor over $X$. Then $E_\rho$ is BGS essential if and only if $\rho$ is spanned by an element of the Hilbert basis of $B$.
\end{prop}

\begin{dfn}
Let $X$ be a normal $\Q$-Gorenstein variety. The \emph{discrepancy} $a(E)$ of a divisor $E$ over $X$ is the coefficient of $E$ in the relative canonical divisor $K_{Y/X}$ associated to a proper birational morphism $Y\to X$ where $E$ appears on $Y$ as a prime divisor. In our setting, each toric divisor $E$ corresponds to a ray $\R_{\geq0}w$ in the fan $\Sigma$ corresponding to $X$. We denote the discrepancy $a(E)$ by $a(w)$.
\end{dfn}

\begin{dfn}
Let $X$ be a normal $\Q$-Gorenstein variety. We say that $X$ has \emph{terminal} (resp. \emph{canonical}) singularities if there exists a proper birational morphism $f\colon Y\to X$ with $Y$ smooth, such that the discrepancy of every exceptional prime divisor of $f$ is positive (resp. non-negative).
\end{dfn}

\begin{rem}\label{ess-disc}
There are close relations between essential divisors and discrepancies. In particular, a toric divisor over $X$ with minimal discrepancy is essential. Moreover, if $X$ is canonical and a toric divisor $E$ over $X$ has discrepancy less than one, then $E$ is essential.
\end{rem}

\subsection{The sufficient conditions of existence of critical arrows}

\begin{prop}[{\cite[Theorem 6.3]{CMDY25}}]
Let $X$ be an affine toric variety and let $w\in B\setminus\{0\}$. If $(\Lambda_{(=1)}^w)^\circ\cap M\neq\emptyset$, then the ray $\R_{\geq0}w$ belongs to $\Delta$.
\end{prop}

We obtain the following result as an analogous statement for each $w$-critical arrow in a simplicial cone.

\begin{thm}\label{polytope} 
Let 
$ v_1,\ldots,v_d \in  N_\R$ and   
$ u_1,\ldots,u_d  \in M_\R$ be bases of $N_\R$ and $M_\R$ that are the dual of each other, and let  $\sigma=\langle v_1,\ldots,v_d\rangle_{\R_{\geq0}} \subset N_\R$ and
$\sigma^\vee=\langle u_1,\ldots,u_d\rangle_{\R_{\geq0}} \subset M_\R$ be the simplicial cones that they generate.  Let $w\in B\setminus\{0\}$ and let $\mu$ be the minimal face of $\sigma$ containing $w$. Let  $ i  \in \{1,\dots, d\}$ be such that $v_{i}\in\mu$. If there exists $u\in M$ satisfying
\begin{equation*}
(u,w)=1,(u,v_{i})>0\text{ and }(u,v_j)\geq0\ (\forall j\neq i),
\end{equation*}
then there exists a $w$-critical arrow $\alpha$ whose tail $\tail_\alpha$ is contained in the ray $\R_{\geq0}u_{i}$.
\end{thm}

\begin{proof}
Let $\delta_i=\langle v_1,\ldots,v_{i-1},-w,v_{i+1},\ldots,v_d\rangle_{\R_{\geq0}}$. We claim that every element of $\delta _i^\vee\cap M\setminus \{0\}$ where the linear function $-v_i$ attains the minimum is contained in $w^\perp$. To see this, let $u' \in \delta _i^\vee\cap M\setminus w^\perp$. Then, $e:=(u',-w)$ is a positive integer. 
Let $u'':=u'+eu$. For every $j\ne i$,  we have
\[
 (u'',v_j) = (u',v_j) + e (u,v_j) \geq 0
\]
and
\[
 (u'',-w) = (u',-w) + e (u,-w) = e + e\cdot (-1) =0.
\]
These imply that $u'' \in \delta_i^\vee\cap M \cap w^\perp$. On the other hand, from the assumption that $(u,-v_i)<0$, we have 
\[
(u'',-v_i) = (u',-v_i) + e (u,-v_i) < (u',-v_i). 
\]
Thus, $-v_i$ does not attain the minimum at $u'$, which shows the claim. 
The claim implies that the finite set $\Crit_i^w$ in Definition \ref{Crit-delta} is not empty. From Proposition \ref{Crit-set}, there exists a $w$-critical arrow $\alpha$ whose tail $\tail_\alpha$ is contained in the ray $\R_{\geq0}u_{i}$.
\end{proof}
\begin{cor}\label{exist-crit}
Let $\sigma$ be a simplicial cone of $N_\R$ and let $w\in B\setminus\{0\}$. If $\Lambda_{(=1)}^w\cap M$ is nonempty, then there exists a $w$-critical arrow.
\end{cor}

\begin{proof}
For every element $u \in \Lambda_{(=1)}^w \cap M$ and for every $1\le j\le d$, we have   $(u,v_j)\ge 0$. Moreover, since $u\ne 0$, at least one of $(u,v_j)$ is strictly positive. Thus, the assumption of Theorem \ref{polytope} holds for some $1\le i \le d$ and the corollary follows.
\end{proof}

\subsection{Three-dimensional $\Q$-factorial terminal toric singularities}\label{3terminal}

Let $\sigma$ be a three-dimensional simplicial terminal cone. By the terminal lemma in \cite[pp.\ 34--36]{Oda88}, we may choose a suitable basis of $N$ such that $\sigma$ is generated by $v_1=(1,0,0),v_2=(0,1,0),v_3=(1,p,q)$ 
with $1\leq p<q$ and $\gcd(p,q)=1$.
\begin{lem}[{\cite[Proposition 2.1]{BGS95}}]
For $1 \le k \le q-1 $, we define $m_k=\min\{l\in\Z_{>0}\mid lq>kp\}$. 
Then, we have
\begin{equation*}
\Hilb_N(\sigma)=\{v_1,v_2,v_3\}\cup\{(1,m_k,k)\mid1\leq k\leq q-1\}.
\end{equation*}
\end{lem}

For each $w=(1,m_k,k)\in\Hilb_N(\sigma)$, it can be written by
\[
w=\frac{q-k}{q}v_1+\left(m_k-\frac{pk}{q}\right)v_2+\frac{k}{q}v_3.
\]
Thus, all elements of the Hilbert basis of $\sigma$ are in the relative interior $\sigma^\circ$ of $\sigma$. It is equivalent to $\mu=\sigma$ for each $w\in\Hilb_N(\sigma)$\cite[Corollary 3.6]{CMDY25}.

\begin{thm}\label{3-dim-terminal}
Let $X$ be a three-dimensional $\Q$-factorial terminal affine toric variety. Then every essential divisor over $X$ appears on $\widetilde{X}$ as a prime divisor.
\end{thm}

\begin{proof}
Let $X$ be a three-dimensional $\Q$-factorial terminal affine toric variety associated to $\sigma$ and let $w=(1,m_k,k)$ be an element of the Hilbert basis of $\sigma$. Let $\delta_1$ be the cone generated by $-w,v_2,v_3$ and let $\delta_3$ be the cone generated by $v_1,v_2,-w$. Let $u=(1,0,0)\in M$. Since $(u,w)=1$, $(u,v_1)=1>0$, $(u,v_2)=0$ and $(u,v_3)=1>0$, Theorem \ref{polytope} implies that there exist two $w$-critical arrows $\alpha$ and $\alpha'$ whose tails $\tail_\alpha$ and $\tail_{\alpha'}$ are in the rays $\R_{\ge 0}u_1$ and $\R_{\ge 0}u_3$, respectively.

Now, we check that $\dim V_w=2$. The computation is similar to that in the proof of Theorem \ref{polytope}. Note that
\begin{align*}
\delta_1^\vee&=\langle(-q,0,1),(pk-qm_k,q-k,m_k-p),(-k,0,1)\rangle_{\R_{\geq0}}, \\
\delta_3^\vee&=\langle(k,0,-1),(0,k',-m_k'),(0,0,-1)\rangle_{\R_{\geq0}},
\end{align*}
where $k'=k/\gcd(k,m_k)$ and $m_k'=m_k/\gcd(k,m_k)$. Hence, we have
\begin{align*}
\delta_1^\vee\cap w^\perp&=\langle (pk-qm_k,q-k,m_k-p),(-k,0,1)\rangle_{\R_{\geq0}}, \\
\delta_3^\vee\cap w^\perp&=\langle(k,0,-1),(0,k',-m_k')\rangle_{\R_{\geq0}}.
\end{align*}
Suppose that $\dim V_w=1$. Then we see that $V_w=\langle (k,0,-1)\rangle_\R$. Since $m_k\leq k$, we have
\[
u'\coloneqq(0,1,-1)=\frac{k'-m_k'}{k'}(0,0,-1)+\frac{1}{k'}(0,k',-m_k')\in\delta_3^{\vee}.
\]
Notice that $(u',-w)=k-m_k$, so we have $u'+(k-m_k)u=(k-m_k,1,-1)\in\delta_3^\vee\cap w^\perp$. 
On the other hand, since $m_k\leq p$, we have
\begin{equation*}
(u'+(k-m_k)u,-v_3)=q-k-(p-m_k)\leq q-k=((k,0,-1),-v_3).
\end{equation*}
Since $u'+(k-m_k)u\notin\langle(k,0,-1)\rangle_\R$, there exists a vector $u''\in\Crit_3^w$ associated with some $w$-critical arrow 
which does not belong to $\langle(k,0,-1)\rangle_\R$. However, this contradicts the assumption that $V_w$ is one-dimensional. Thus, we obtain $\dim V_w=2$. 
Therefore, by Corollary \ref{min}, we obtain that the essential divisor associated with $w$ appears on $\widetilde{X}$ as a prime divisor.
\end{proof}
\begin{rem}\label{Kedzierski}
In \cite[Lemma 3.13]{Ked11}, it is proved that the $G$-Hilbert scheme associated to a three-dimensional $\Q$-factorial terminal affine toric variety has only conifold singularities, which are locally defined by equation $xy-zw$. 
This shows that the $G$-Hilbert scheme admits a small resolution, on which every essential divisor over the original variety appears by definition. Thus, we see that every essential divisor appears on the $G$-Hilbert scheme as well. This argument gives an alternative proof of the above theorem. 
\end{rem}

\subsection{Three-dimensional canonical abelian quotient singularities}\label{3can}

It is well known that three-dimensional $\Q$-factorial terminal toric singularities are cyclic quotient singularities; these were treated in the previous section. In this section, we study canonical abelian quotient singularities, which form a strictly larger class.

We begin by recalling some general notation. For later use, consider a finite (not necessarily cyclic) abelian group $G\subset\GL(d,k)$. Suppose that the characteristic of $k$ does not divide the order of $G$. Then all elements of $G$ are simultaneously diagonalizable. Thus, without loss of generality, we assume that it belongs to the group of diagonal matrices. We denote by $\overline{g}$ the element of $N$ corresponding to the diagonal matrix $g\in\GL(d,k)$. Let
\begin{equation*}
N\coloneqq\Z^d+\sum_{g\in G}\Z\overline{g}.
\end{equation*}
Define
\begin{equation*}
M\coloneqq N^\vee=\{x\in\R^d\mid\forall y\in N,(x,y)\in\Z\}\subset\Z^d.
\end{equation*}
Note that, to verify whether an element belongs to $M$, it suffices to consider elements of the form $z+g$ with $z\in\Z^d$ and $g$ a generator of $G$. Let $\sigma=(\R_{\geq0})^d$ be a region of $N_\R=N\otimes\R$ whose all entries are non-negative. Then the toric variety determined by $\sigma$ is isomorphic to $\A_k^d/G$.

The classification of three-dimensional canonical cyclic quotient singularities is given by Ishida and Iwashita \cite{II86} (see also \cite{SS23}).

\begin{prop}[\cite{II86}]\label{class-3can}
Let $G$ be a finite cyclic subgroup of $\GL(3,\C)$. Then $\A_\C^3/G$ has only canonical singularity if and only if
\begin{enumerate}
\item $G$ is of type $\frac{1}{r}(a,b,c)$ with $a+b+c\equiv0 \mod r$,
\item $G$ is of type $\frac{1}{r}(1,a,r-a)$ with $\gcd(a,r)=1$ and $a<r$,
\item $G$ is of type $\frac{1}{r}(1,r-1,a)$ with $\gcd(a,r)>1$ and $a<r$,
\item $G$ is of type $\frac{1}{4k}(1,2k+1,4k-2)$ with $k\geq2$ or
\item $G$ is of type $\frac{1}{9}(1,4,7)$ or $\frac{1}{14}(1,9,11)$.
\end{enumerate}
\end{prop}
Case (2) is a terminal singularity, which was discussed in Section \ref{3terminal}.

\begin{rem}\label{expert}
It is well known to experts, and proved in \cite{II86}, that the three-dimensional non-Gorenstein canonical abelian quotient singularities are exactly those of types (2)--(5) in the above classification. Thus, let $G$ be a finite abelian subgroup of $\GL(3,\C)$. Then $\A_\C^3/G$ has only canonical singularity if and only if $G\subset\SL(3,\C)$ or $G$ is of type (2)--(5).
\end{rem}

\begin{rem}
Let $G$ be a cyclic group generated by $\frac{1}{r}(a,b,c)$ and let $w_l=\frac{1}{r}(\overline{la},\overline{lb},\overline{lc})$ for $1\leq l\leq r-1$. Then, we see that
\[
\Hilb_N(\sigma)\subset\{e_1,e_2,e_3\}\cup\left\{w_l\mid1\leq l\leq r-1\right\}.
\]
Here, $e_1,e_2,e_3$ are standard basis of $\R^3$ and $\overline{a}$ denotes the unique integer $\overline{a}\in\{0,\ldots,r-1\}$ such that $\overline{a}\equiv a\mod r$. Each point on the left-hand side corresponds to some power of the generator of $G$.
\end{rem}

\begin{thm}\label{3-can}
Let $X=\A^3/G$ be a three-dimensional canonical cyclic quotient singularity and let $w$ be an element of the Hilbert basis of $\sigma\cap N$. Then the set $\Lambda_{(=1)}^w\cap M$ is empty if and only if $G$ is of type $\frac{1}{14}(1,9,11)$ with $w=\frac{1}{14}(7,7,7)$.
\end{thm}
\begin{proof}
The proof proceeds case by case according to the above classification (see Proposition \ref{class-3can}).

For case (5), we can check the claim with direct computation, both by hand and by computer. The computer calculations are contained in the accompanying file. (We do not include this computation, since it is too long to be included.)  Hence, it suffices to show that $\Lambda_{(=1)}^w\cap M$ is nonempty in cases (1)--(4).

Since case (1) admits a crepant resolution \cite{Nak01, BKR01}, it follows from \cite[Lemma 6.4 and Example 6.12]{CMDY25} that $\Lambda_{(=1)}^w\cap M$ is nonempty.

Case (2) was already treated in Section \ref{3terminal}, where we proved that $\Lambda_{(=1)}^w\cap M$ is nonempty.

In case (3), for any $1\leq l,l'\leq r-1$ such that $l\neq l'$, since
\[
w_l+w_{l'}=\frac{1}{r}(l+l',2r-(l+l'),\overline{a(l+l')})
\]
and either $l+l'$ or $2r-(l+l')$ greater than $r$, each $w_l=\frac{1}{r}(l,r-l,\overline{al})$ does not decompose into two elements of $B\setminus\{0\}$. Thus, we have
\[
\Hilb_N(\sigma)=\{e_1,e_2,e_3\}\cup\left\{\frac{1}{r}(l,r-l,\overline{al})\mid1\leq l\leq r-1\right\}.
\]
Let $w_l=\frac{1}{r}(l,r-l,\overline{al})$ and let $u=(1,1,0)\in\sigma^\vee$. Then
\begin{equation*}
\left(u,\frac{1}{r}(1,r-1,a)\right)=1\in\Z,
\end{equation*}
so $u\in M$. Since $(u,w_l)=1$, we have $u\in \Lambda_{(=1)}\cap M$.

For case (4), the Hilbert basis of $B$ is precisely
\[
\Hilb_N(\sigma)=\{e_1,e_2,e_3\}\cup\{w_l\mid2\nmid l\text{ or }1\leq l\leq 2k\text{ and }2\mid l\}.
\]
First, we show that the right hand side is contained in the Hilbert basis of $B$. Recall that if a toric variety is canonical, then the discrepancy $a(w)<1$ implies that $w\in\Hilb_N(\sigma)$ (see Remark \ref{ess-disc}). Let $w_l=\frac{1}{4k}(\overline{l},\overline{l(2k+1)},\overline{l(4k-2)})$ be an element of the right hand side. If $2\nmid l$ and $1\leq l\leq2k$, then $w_l=\frac{1}{4k}(l,2k+l,4k-2l)$. Since $a(w)=1/2$, we have $w_l\in\Hilb_N(\sigma)$. If $2\nmid l$ and $2k+1\leq l\leq4k-1$, then $w_l=\frac{1}{4k}(l,l-2k,8k-2l)$. Since $a(w)=1/2$, we have $w_l\in\Hilb_N(\sigma)$. If $1\leq l\leq 2k$ and $2\mid l$, then $w_l=\frac{1}{4k}(l,l,4k-2l)$. Since $a(w)=0$, we have $w_l\in\Hilb_N(\sigma)$. Thus, the right hand side is contained in the Hilbert basis of $B$. We now show that other elements are not in the Hilbert basis of $\sigma$. Let $w_l=\frac{1}{4k}(\overline{l},\overline{l(2k+1)},\overline{l(4k-2)})$ where $2\mid l$ and $2k+1\leq l\leq 4k-1$. Then $w_l=\frac{1}{4k}(l,l,8k-2l)$. If $l=4m$ for some $m\in\Z_{>0}$, then
\begin{align*}
\frac{1}{4k}(l,l,8k-2l)&=\frac{1}{4k}(2m,2m,4k-2\cdot2m)+\frac{1}{4k}(2m,2m,4k-2\cdot2m) \\
&=w_{l/2}+w_{l/2}.
\end{align*}
If $l=4m-2$ for some $m\in\Z_{>0}$, then
\begin{align*}
\frac{1}{4k}(l,l,8k-2l)&=\frac{1}{4k}(2m,2m,4k-2\cdot2m)+\frac{1}{4k}(2m-2,2m-2,4k-2(2m-2)) \\
&=w_{(l+2)/2}+w_{(l-2)/2}.
\end{align*}
Hence $w$ is decomposable, and since the elements of Hilbert basis are indecomposable. Therefore, they does not belong to the Hilbert basis of $B$.

We now show that $\Lambda_{(=1)}^w\cap M\neq\emptyset$ for all $w\in\Hilb_N(\sigma)$. Let
\[
w_l=\frac{1}{4k}(\overline{l},\overline{l(2k+1)},\overline{l(4k-2)}).
\]
If $2\nmid l$ and $1\leq l\leq 2k$, then $w_l=\frac{1}{4k}(l,2k+l,4k-2l)$ and hence $u=(2,0,1)\in\Lambda_{(=1)}\cap M$. If $2\nmid l$ and $2k+1\leq l\leq4k-1$, then $w_l=\frac{1}{4k}(l,l-2k,8k-2l)$ and hence $u=(0,2,1)\in\Lambda_{(=1)}\cap M$. If $2\mid l$ and $1\leq l\leq2k$, then $w_l=\frac{1}{4k}(l,l,4k-2l)$ and hence $u=(0,2,1)\in\Lambda_{(=1)}\cap M$. Therefore, $\Lambda_{(=1)}^w\cap M\neq\emptyset$ for all $w\in\Hilb_N(\sigma)$.
\end{proof}

\begin{cor}\label{3-can-div}
Let $X=\A^3/G$ be a three-dimensional canonical abelian quotient singularity. Then every essential divisor over $X$ has center of positive dimension on $\widetilde{X}$.
\end{cor}

\begin{proof}
The proof is partly computer-assisted. By Remark \ref{expert}, it suffices to verify the assertion for the types appearing in the classification of Proposition \ref{class-3can}. When $G$ is not of type $\frac{1}{14}(1,9,11)$ with $w=\frac{1}{14}(7,7,7)$, the assertion follows from  Theorem \ref{3-can}, Corollaries \ref{min} and \ref{exist-crit}. When  $G$ is of type $\frac{1}{14}(1,9,11)$ with $w=\frac{1}{14}(7,7,7)$,  performing the algorithm in Section \ref{pseudo_code} say with SageMath, we see that there are two $w$-critical arrows which are linearly independent. Therefore, by Corollary \ref{min}, the ray $\R_{\geq0}w$ belongs to $\Delta$.
\end{proof}

\subsection{Numerical experiments on three-dimensional $\Q$-factorial toric varieties.}

By a computer-assisted computation, we obtain the following.

\begin{thm}\label{many-example}
Let $\sigma$ be a simplicial cone generated by $(1,0,0)$, $(x_1,y_1,z_1)$, $(x_2,y_2,z_2)$ with $0\leq x_i,y_i,z_i\leq3$ and $y_i\leq z_i$ for $i=1,2$. Then every essential divisor over $X$ appears on $\widetilde{X}$ as a prime divisor.
\end{thm}

\begin{exam}
Theorem \ref{many-example} contains the following examples.
\begin{enumerate}
\item (cyclic non-canonical case) The cone $\sigma$ generated by $(1,0,0)$, $(1,1,3)$, $(2,2,3)$.
\item (non-cyclic non-canonical case) The cone $\sigma$ generated by $(1,0,0)$, $(2,0,3)$, $(2,3,3)$.
\end{enumerate}
\end{exam}

\subsection{Examples of four-dimensional toric varieties which are negative answers to Problem \ref{FB-ess}}

\begin{exam}\label{example}
The following examples are four-dimensional cones that admit an element $w$ of the Hilbert basis whose minimal cone $\tau\subset\sigma$ is of dimension two. We compute such examples by using SageMath \cite{sagemath} (see Section \ref{pseudo_code}).
\begin{table}[h]
\centering
\label{tab:hogehoge}
\begin{tabular}{|c|c|c|c|}
\hline
generators of $\sigma$ & $w$ & $\dim\tau$ & $a(w)$ \\ \hline
$(1,0,0,0),(0,1,0,0),(1,3,3,0),(4,3,1,4)$ & $(4,4,2,3)$ & 2 & $5/4$ \\ \hline
$(1,0,0,0),(0,1,0,0),(4,1,4,0),(3,4,1,4)$ & $(4,4,2,3)$ & 2 & $5/4$ \\ \hline
$(1,0,0,0),(1,3,1,1),(3,2,3,1),(3,2,2,3)$ & $(6,6,5,4)$ & $2$ & $9/5$ \\ \hline
$(1,0,0,0),(0,1,2,2),(1,2,1,0),(2,3,1,3)$ & $(3,5,3,4)$ & $2$ & $19/11$ \\ \hline
$(1,0,0,0),(0,1,3,1),(3,3,1,0),(1,3,3,3)$ & $(2,3,4,2)$ & $2$ & $21/18$ \\ \hline
$(1,0,0,0),(0,1,3,1),(2,2,1,0),(1,3,2,3)$ & $(2,3,4,2)$ & $2$ & $10/7$ \\ \hline
$(1,0,0,0),(0,1,3,3),(3,1,3,0),(2,3,2,1)$ & $(3,3,4,2)$ & $2$ & $6/7$ \\ \hline
$(1,0,0,0),(0,1,3,3),(3,1,3,0),(1,3,1,1)$ & $(2,3,3,2)$ & $2$ & $3/4$ \\ \hline
\end{tabular}
\end{table}

\end{exam}
We see that the first five examples admit a Hilbert basis resolution (see the definition in \cite{SS23,KY26}) using Macaulay2 \cite{NormalToricVarietiesSource}. We do not know whether other examples admit a Hilbert basis resolution. On the other hand, these examples suggest that there is a relation between the appearance of $\R_{\geq0}w$ and the discrepancy of the corresponding divisor. Thus, we propose the following problem.
\begin{prob}
Let $\sigma$ be a simplicial cone and $w\in B\setminus\{0\}$. Suppose that $\sigma$ admits a Hilbert basis resolution. Do we have $a(w)>1-1/|G|$ whenever the ray $\R_{\geq0}w$ does not belong to $\Delta$ ?
\end{prob}

\section{The volume of the level 1 subpolyhedron.}\label{abelian}

In this section, we consider the existence of $w$-critical arrows from the viewpoint of volume in the setting of abelian quotient singularities.
\begin{lem}\label{epsilon}
For $0<\varepsilon<1$, every integral arrow $\alpha$ in $\Lambda_{(\leq1-\varepsilon)}^w$ such that $\vect_\alpha\in w^\perp$ and the tail $\tail$ is a vertex of $\Lambda_{(\leq1-\varepsilon)}^w$ is a $w$-critical arrow.
\end{lem}

\begin{proof}
Let $\alpha=(\head,\tail)$ be an integral arrow in $\Lambda_{(\leq1-\varepsilon)}^w$ such that $\vect_\alpha\in w^\perp$ and the tail $\tail$ is a vertex of $\Lambda_{(\leq1-\varepsilon)}^w$. If there exists a point $P\in\Lambda_{(\leq1-\varepsilon)}^w$ such that $w(P)<w(\tail)$ and $P-\tail\in M$, then since $w(P)\geq0$ and $w(\tail)\leq1-\varepsilon$, we have
\[
-1<-1+\varepsilon\leq-w(\tail)\leq w(P)-w(\tail)<0.
\]
However, we have $w(P)-w(\tail)\in\Z$ which contradicts the above inequality. Therefore, $\alpha$ is a $w$-critical arrow.
\end{proof}

We use the notation introduced in Section \ref{3can}. In particular, $G\subset\GL(d,k)$ is a finite abelian subgroup, $N$ is a lattice $\Z^d+\sum_{g\in G}\Z\overline{g}$ and $\sigma=(\R_{\geq0})^d$ is a region of $N_\R$ whose all entries are non-negative.

Let
\begin{equation*}
w=\frac{1}{l}(a_1,\ldots,a_d)\in\sigma^\circ\cap N,w'=\frac{l}{a_1^2+\cdots+a_d^2}(a_1,\ldots,a_d)
\end{equation*}
so that $(w',w)=1$. We refer to
\begin{align*}
\Lambda_{(=1)}^w&\coloneqq\{(x_1,\ldots,x_d)\in\R^d\mid x_i\geq0,(-,w)=1\}, \\
\Lambda_{(\leq1)}^w&\coloneqq\{(x_1,\ldots,x_d)\in\R^d\mid x_i\geq0,(-,w)\leq1\}
\end{align*}
as the level 1 subpolyhedron and level $\leq1$ subpolyhedron, respectively.

\begin{prop}
We have
\begin{equation*}
\Vol_{M\cap w^\perp}(\Lambda_{(=1)}^w)=\frac{l^d}{(d-1)!a_1\cdots a_d|G|}.
\end{equation*}
Here the left hand side is defined to be the volume of $\Lambda_{(=1)}^w-w'\subset w^\perp$ with respect to the measure on $w^\perp$ normalized by the lattice $M\cap w^\perp$, and $|G|$ denotes the order of $G$.
\end{prop}

\begin{proof}
The hyperplane $(-,w)=1$ is given by
\begin{equation*}
a_1x_1+\cdots+a_dx_d=l.
\end{equation*}
Thus the simplex $\Lambda_{(\leq1)}^w$ has the vertices $0,\frac{l}{a_1}e_1,\ldots,\frac{l}{a_d}e_d$. Therefore,
\begin{align*}
\Vol_{\Z^d}(\Lambda_{(\leq1)}^w)&=\Vol_{\Z^d}(\text{the $d$-simplex with vertices $0,e_1,\ldots,e_d$})\times\frac{l}{a_1}\cdots\frac{l}{a_d} \\
&=\frac{l^d}{d!a_1\cdots a_d}.
\end{align*}
Here, $\Vol_{\Z^d}$ is the measure normalized by the standard lattice $\Z^d$, that is, the usual Lebesgue measure of $\R^d$. On the other hand, since $\Lambda_{(\leq1)}^w$ is a cone over $\Lambda_{(=1)}^w$ of height $|w'|=|w|^{-1}$, we have
\begin{equation*}
\Vol_{\Z^d}(\Lambda_{(\leq1)}^w)=\frac{1}{d}|w'|\Vol_{\Z^d}(\Lambda_{(=1)}^w).
\end{equation*}
Hence,
\begin{align*}
\Vol_{\Z^d}(\Lambda_{(=1)}^w)&=d|w|\Vol_{\Z^d}(\Lambda_{(\leq1)}^w) \\
&=d\frac{\sqrt{a_1^2+\cdots+a_d^2}}{l}\frac{l^d}{d!a_1\cdots a_d} \\
&=\frac{l^{d-1}\sqrt{a_1^2+\cdots+a_d^2}}{(d-1)!a_1\cdots a_d}.
\end{align*}
Next, we transform $\Vol_{\Z^d}(\Lambda_{(=1)}^w)$ to $\Vol_{M\cap w^\perp}(\Lambda_{(=1)}^w)$. By \cite[Corollary 1.3.5]{Mar03},
\begin{equation*}
\det(N)\det(M\cap w^\perp)=\det(\R w\cap N)=\det(\Z w).
\end{equation*}
Since
\[
\det(M\cap w^\perp)=|G||w|=\frac{|G|}{l}\sqrt{a_1^2+\cdots+a_d^2},
\]
we obtain
\begin{equation*}
\Vol_{M\cap w^\perp}(\Lambda_{(=1)}^w)=\frac{\Vol_{\Z^d}(\Lambda_{(=1)}^w)}{\det(M\cap w^\perp)}=\frac{l^d}{(d-1)!a_1\cdots a_d|G|}.
\end{equation*}
\end{proof}

\begin{thm}\label{volume}
If $\Vol_{M\cap w^\perp}(\Lambda_{(=1)}^w)>1$, then there exists at least one $w$-critical arrow in $A_\R$.
\end{thm}

\begin{proof}
Suppose that $\Vol_{M\cap w^\perp}(\Lambda_{(=1)}^w)>1$. Let
\begin{equation*}
\Lambda_{(\leq1-\varepsilon)}^w\coloneqq\{x\mid x_1\geq0,\ldots,x_d\geq0,(-,w)=1-\varepsilon\}\ (0<\varepsilon\ll1).
\end{equation*}
Since $\Vol_{M\cap w^\perp}(\Lambda_{(\leq1-\varepsilon)}^w)>1$, from Blichfeldt's theorem \cite{Bli14}, there exist two points $\head,\tail\in \Lambda_{(\leq1-\varepsilon)}^w$ such that $\head-\tail\in M\cap w^\perp$. Namely, $\Lambda_{(\leq1-\varepsilon)}^w$ contains an integral arrow. We may assume that the tail $\tail$ is a vertex of $\Lambda_{(\leq1-\varepsilon)}^w$. Then, from Lemma \ref{epsilon}, the arrow $(\head,\tail)$ is a $w$-critical arrow in $A_\R$. 
\end{proof}
We now consider whether the volume is greater than one. We see that
\begin{align}
&\Vol_{M\cap w^\perp}(\Lambda_{(=1)}^w)=\frac{l^d}{(d-1)!a_1\cdots a_d|G|}>1 \notag \\
&\Leftrightarrow\frac{a_1\cdots a_d}{l^d}<\frac{1}{|G|(d-1)!} \label{vol_log}\\
&\Leftrightarrow\sum_{i=1}^d\log\frac{a_i}{l}<-\log|G|(d-1)!. \notag
\end{align}
Moreover, we also obtain
\begin{align*}
\eqref{vol_log}&\Leftrightarrow\sqrt[d]{\frac{a_1\cdots a_d}{l^d}}<\sqrt[d]{\frac{1}{|G|(d-1)!}} \\
&\Leftarrow\frac{1}{d}\sum_{i=1}^d\frac{a_i}{l}<\sqrt[d]{\frac{1}{|G|(d-1)!}} \\
&\Leftrightarrow\sum_{i=1}^d\frac{a_i}{l}<\sqrt[d]{\frac{d^d}{|G|(d-1)!}}=d\sqrt[d]{\frac{1}{|G|(d-1)!}}\eqqcolon H_{d,|G|}.
\end{align*}
We now compute several values of $H_{d,|G|}$. For a fixed $d$, the function $H_{d,|G|}$ is decreasing in $|G|$.
\begin{equation*}
H_{3,13}=1.0126\cdots,H_{4,42}=1.00394\cdots,H_{5,130}=1.00032\cdots,\ldots
\end{equation*}
\begin{cor}
If $|G|\leq13$ for $d=3$, $|G|\leq42$ for $d=4$, or $|G|\leq130$ for $d=5$, then every toric divisor $E$ over $X$ associated to $w\in\sigma^\circ\cap N$ with $a(E)\leq0$ has center of positive dimension on $\widetilde{X}$.
\end{cor}

\bibliographystyle{alpha}
\bibliography{critical_arrow}

\end{document}